\documentclass{amsart}

\usepackage{amsmath,amssymb,amsfonts,enumerate,amsthm,graphicx,color}
\usepackage{amscd,amsmath,amsthm,amssymb}
\usepackage{pstcol,pst-plot,pst-3d}
\usepackage{lmodern,pst-node}
\usepackage{lineno}

\def\opn#1#2{\def#1{\operatorname{#2}}} 
\opn\chara{char} \opn\length{\ell} \opn\pd{pd} \opn\rk{rk}
\opn\projdim{proj\,dim} \opn\injdim{inj\,dim} \opn\rank{rank}
\opn\depth{depth} \opn\grade{grade} \opn\height{height}
\opn\embdim{emb\,dim} \opn\codim{codim}

\opn\Tr{Tr} \opn\bigrank{big\,rank}
\opn\superheight{superheight}\opn\lcm{lcm}
\opn\trdeg{tr\,deg}%
\opn\reg{reg} \opn\lreg{lreg} \opn\skel{skel}
\opn\multideg{multideg}
\opn\div{div} \opn\Div{Div} \opn\cl{cl} \opn\Cl{Cl}
\opn\Spec{Spec} \opn\Supp{Supp} \opn\supp{supp} \opn\Sing{Sing}
\opn\Ass{Ass}
\opn\Ann{Ann} \opn\Rad{Rad} \opn\Soc{Soc}
\opn\Ker{Ker} \opn\Coker{Coker} \opn\Im{Im} \opn\Hom{Hom}
\opn\Tor{Tor} \opn\Ext{Ext} \opn\End{End} \opn\Aut{Aut}
\opn\id{id}

\opn\nat{nat}
\opn\pff{pf}
\opn\Pf{Pf} \opn\GL{GL} \opn\SL{SL} \opn\mod{mod} \opn\ord{ord}
\opn\aff{aff} \opn\con{conv} \opn\relint{relint} \opn\st{st}
\opn\lk{lk} \opn\cn{cn} \opn\core{core} \opn\vol{vol}
\opn\link{link} \opn\star{star} \opn\skel{skel} \opn\Reg{Reg}
\opn\gr{gr}

\def\pot#1#2{#1[\kern-0.28ex[#2]\kern-0.28ex]}

\opn\dirlim{\underrightarrow{\lim}}
\opn\inivlim{\underleftarrow{\lim}}
\def\Implies{\ifmmode\Longrightarrow \else
     \unskip${}\Longrightarrow{}$\ignorespaces\fi}
\def\implies{\ifmmode\Rightarrow \else
     \unskip${}\Rightarrow{}$\ignorespaces\fi}
\def\iff{\ifmmode\Longleftrightarrow \else
     \unskip${}\Longleftrightarrow{}$\ignorespaces\fi}

\let\:=\colon

\newtheorem{thm}{Theorem}[section]
\newtheorem{cor}[thm]{Corollary}
\newtheorem{lem}[thm]{Lemma}
\newtheorem{prop}[thm]{Proposition}
\newtheorem{defn}[thm]{Definition}

\newtheorem{rem}[thm]{Remark}

\numberwithin{equation}{section}

\begin{document}

\bibliographystyle{amsplain}

\title{ The regularity of binomial edge ideals of graphs }
\author{ Dariush Kiani and Sara Saeedi Madani }
\thanks{2010 \textit{Mathematics Subject Classification.} 05E40, 05C25, 16E05, 13C05}
\thanks{\textit{Key words and phrases.} Binomial edge ideal, Castelnuovo-Mumford regularity. }

\address{Dariush Kiani, Department of Pure Mathematics,
 Faculty of Mathematics and Computer Science,
 Amirkabir University of Technology (Tehran Polytechnic),
424, Hafez Ave., Tehran 15914, Iran, and School of Mathematics, Institute for Research in Fundamental Sciences (IPM),
P.O. Box 19395-5746, Tehran, Iran.} \email{dkiani@aut.ac.ir, dkiani7@gmail.com}
\address{Sara Saeedi Madani, Department of Pure Mathematics,
 Faculty of Mathematics and Computer Science,
 Amirkabir University of Technology (Tehran Polytechnic),
424, Hafez Ave., Tehran 15914, Iran, and School of Mathematics, Institute for Research in Fundamental Sciences (IPM),
P.O. Box 19395-5746, Tehran, Iran.} \email{sarasaeedi@aut.ac.ir}

\begin{abstract}
We prove two recent conjectures on some upper bounds for the Castelnuovo-Mumford regularity of the binomial edge ideals of some different classes of graphs. We prove the conjecture of Matsuda and Murai for graphs which has a cut edge or a simplicial vertex, and hence for chordal graphs. We determine the regularity of the binomial edge ideal of the join of graphs in terms of the regularity of the original graphs, and consequently prove the conjecture of Matsuda and Murai for such a graph, and hence for complete $t$-partite graphs. We also generalize some results of Schenzel and Zafar about complete $t$-partite graphs. We also prove the conjecture due to the authors for a class of chordal graphs.
\end{abstract}

\maketitle

\section{ Introduction }\label{Introduction}

\noindent The binomial edge ideal of a graph was introduced in \cite{HHHKR}, and \cite{O} at about the same time.
Let $G$ be a finite simple graph with vertex set $[n]$ and edge set $E(G)$.
Also, let $S=K[x_1,\ldots,x_n,y_1,\ldots,y_n]$ be the polynomial ring over a field $K$. Then the \textbf{binomial edge ideal} of $G$ in $S$,
denoted by $J_G$, is generated by binomials $f_{ij}=x_iy_j-x_jy_i$, where $i<j$ and $\{i,j\}\in E(G)$. Also, one could see this ideal as an ideal
generated by a collection of 2-minors of a $(2\times n)$-matrix whose entries are all indeterminates. Many of the algebraic properties of such ideals
were studied in \cite{D}, \cite{EHH}, \cite{EZ}, \cite{HHHKR}, \cite{SK}, \cite{SZ}, \cite{Z} and \cite{ZZ}. In \cite{EHHQ}, the authors introduced the binomial edge ideal of a pair of graphs, as a
generalization of the binomial edge ideal of a graph. Let $G_1$ be a graph on the vertex set $[m]$ and $G_2$ a graph on the vertex set $[n]$, and
let $X= (x_{ij})$ be an $(m\times n)$-matrix of indeterminates. Let
$S=K[X]$ be the polynomial ring in the variables $x_{ij}$, where $i=1,\ldots,m$ and $j=1,\ldots,n$. Let $e=\{i,j\}$ for some $1\leq i < j\leq m$
and $f=\{t, l\}$ for some $1\leq t < l\leq n$. To
the pair $(e,f)$, the following $2$-minor of $X$ is assigned:
\[p_{e,f}=[i,j|t,l]=x_{it}x_{jl}-x_{il}x_{jt}.\] Then, the ideal \[J_{G_1,G_2}=(p_{e,f}:~e\in E(G_1), f\in E(G_2))\] is called
the \textbf{binomial edge ideal of the pair} $(G_1,G_2)$. Some properties of this ideal were studied in \cite{SK1}. Note that if $G_1$ is just an edge, then $J_{G_1,G_2}$ is isomorphic to $J_{G_2}$.

In \cite{SK1}, the authors posed a question about the Castelnuovo-Mumford regularity of the binomial edge ideal of graphs, which is if $\mathrm{reg}(J_G)\leq c(G)+1$, where $c(G)$ is the number of maximal cliques of $G$. They also answered this question for closed graphs, in \cite{SK}. In \cite{MM}, the authors gained an upper bound for the regularity of the binomial edge ideal of a graph on $n$ vertices. They showed that $\mathrm{reg}(J_G)\leq n$. They also posed a conjecture which is $\mathrm{reg}(J_G)\leq n-1$, whenever $G$ is not $P_n$, the path over $n$ vertices. In this paper, we investigate about the regularity of the binomial edge ideals and especially these two problems. This paper is organized as follows. In Section~\ref{Preliminaries}, we pose some definitions, facts and notation which will be used throughout the paper. In Section~\ref{The Castelnuovo-Mumford regularity of binomial edge ideals}, we show that the conjecture of Matsuda and Murai is true for every graph which has a cut edge or a simplicial vertex. Hence, we prove this conjecture for all chordal graphs. We also show that the regularity of the binomial edge ideal of the join of two graphs $G_1$ and $G_2$ (not both complete) is equal to $\mathrm{max}\{\mathrm{reg}(J_{G_1}),\mathrm{reg}(J_{G_2}),3\}$. Applying this fact, we prove the conjecture of Matsuda and Murai, in the case of join of graphs, and hence for complete $t$-partite graphs. Then, we generalize some results of Schenzel and Zafar about complete $t$-partite graphs. Using a similar discussion as one appeared in \cite{EZ}, we prove that the conjecture due to the authors is true for a class of chordal graphs including block graphs, which we call them "generalized block graphs". Hence, we extend the recent results of Ene and Zarojanu about block graphs.

Throughout the paper, we mean by a graph $G$, a simple graph. Moreover, if $V=\{v_1,\ldots,v_n\}$ is the vertex set of $G$ (which contains $n$ elements), then, for simplicity, we denote it by $[n]$. 

\section{ Preliminaries }\label{Preliminaries}

\noindent In this section, we review some notions and facts around graphs and binomial ideals associated to graphs, which we need throughout.

In \cite{HHHKR}, the authors determined all graphs whose binomial edge ideal have a Gr\"{o}bner basis with respect to the lexicographic order induced by $x_1>\cdots >x_n>y_1>\cdots >y_n$, and called this class of graphs, \textbf{closed graphs}. There are also some combinatorial descriptions for these graphs, for example in \cite{EHH}, the authors proved that a graph $G$ is closed if and only if there exists a labeling of $G$ such that all facets of $\Delta(G)$ are intervals $[a,b]\subseteq [n]$. Here, $\Delta(G)$ is the clique complex of $G$, the simplicial complex whose facets are the vertex sets of the maximal cliques of $G$. We say that a vertex of $G$ is a \textbf{free vertex}, if it is a free vertex of the simplicial complex $\Delta(G)$, i.e. it is just contained in one facet of $\Delta(G)$. This is equivalent to say that such a vertex $v$ has this property that $N_G(v)$ induces a complete subgraph of $G$, and the vertex $v$ is also called a \textbf{simplicial vertex}. By $N_G(v)$, we mean the set of all neighbors (i.e. adjacent vertices) of the vertex $v$ in $G$.

Let $G$ and $H$ be two graphs on $[m]$ and $[n]$, respectively. We denote by $G*H$, the \textit{join} (product) of two graphs $G$ and $H$, that is
the graph with vertex set $[m]\cup [n]$, and the edge set $E(G)\cup E(H)\cup \{\{v,w\}~:~v\in [m],~w\in [n]\}$. In particular, the cone of a vertex $v$ on a graph $G$ is defined to be their join, that is $v*G$, and is sometimes denoted by $\mathrm{cone}(v,G)$. Let $V$ be a set. To simplify our notation throughout this paper, we introduce the join of two collection of subsets of $V$, $\mathcal{A}$ and $\mathcal{B}$, denoted by $\mathcal{A}\circ \mathcal{B}$, as $\{A\cup B: A\in \mathcal{A}, B\in \mathcal{B}\}$. If $\mathcal{A}_1,\ldots,\mathcal{A}_t$ are collections of subsets of $V$, then we denote their join, by $\bigcirc_{i=1}^{t}\mathcal{A}_i$.

Let $G$ be a graph and $e=\{v,w\}$ an edge of it. Then vertices $v$ and $w$ are called the endpoints of $e$. If $\{e_1,\ldots,e_t\}$ is a set of edges of $G$, then by $G\setminus \{e_1,\ldots,e_t\}$, we mean the graph on the same vertex set as $G$ in which the edges $e_1,\ldots,e_t$ are omitted. Here, we simply write $G\setminus e$, instead of $G\setminus \{e\}$. An edge $e$ of $G$ whose deletion from the graph, implies a graph with more connected components than $G$, is called a \textbf{cut edge} of $G$. Now, we recall a notation from \cite{MSh}. If $v,w$ are two vertices of a graph $G=(V,E)$ and $e=\{v,w\}$ is not an edge of $G$, then $G_e$ is defined to be the graph on the vertex set $V$, and the edge set $E \cup \{\{x,y\}~:~x,y\in N_G(v)~\mathrm{or}~x,y\in N_G(w)\}$. 

A vertex $v$ of $G$ whose deletion from the graph, implies a graph with more connected components than $G$, is called a \textbf{cut point} of $G$. A \textbf{nonseparable} graph is a connected and nontrivial graph with no cut points. A \textbf{block} of a graph is a maximal nonseparable subgraph of it. A \textbf{block graph} is a connected graph whose blocks are complete graphs (see \cite{H} for more information in this topic). A disconnected graph is called a block graph, if all of its connected components are block graphs. One can see that a graph $G$ is a block graph if and only if it is a chordal graph in which every two maximal cliques have at most one vertex in common. This class was considered in~\cite[Theorem~1.1]{EHH}. Here, we introduce a class of chordal graphs including block graphs: Let $G$ be a connected chordal graph such that for every three maximal cliques of $G$ which have a nonempty intersection, the intersection of each pair of them is the same. In other words, $G$ has the property that for every $F_i,F_j,F_k\in \Delta(G)$, if $F_i\cap F_j\cap F_k\neq \emptyset$, then $F_i\cap F_j=F_i\cap F_k=F_j\cap F_k$. We call $G$, a \textbf{generalized block} graph. A disconnected graph is called a generalized block graph, if all of its connected components are generalized block graphs. By the above, it is clear that a block graph is also a generalized block graph. Hence, a tree is also a generalized block graph. The graph depicted in Figure~\ref{Generalized} is a generalized block graph, which is not a block graph.
\begin{center}
\begin{figure}
\hspace{0 cm}
\includegraphics[height=2.1cm,width=3.2cm]{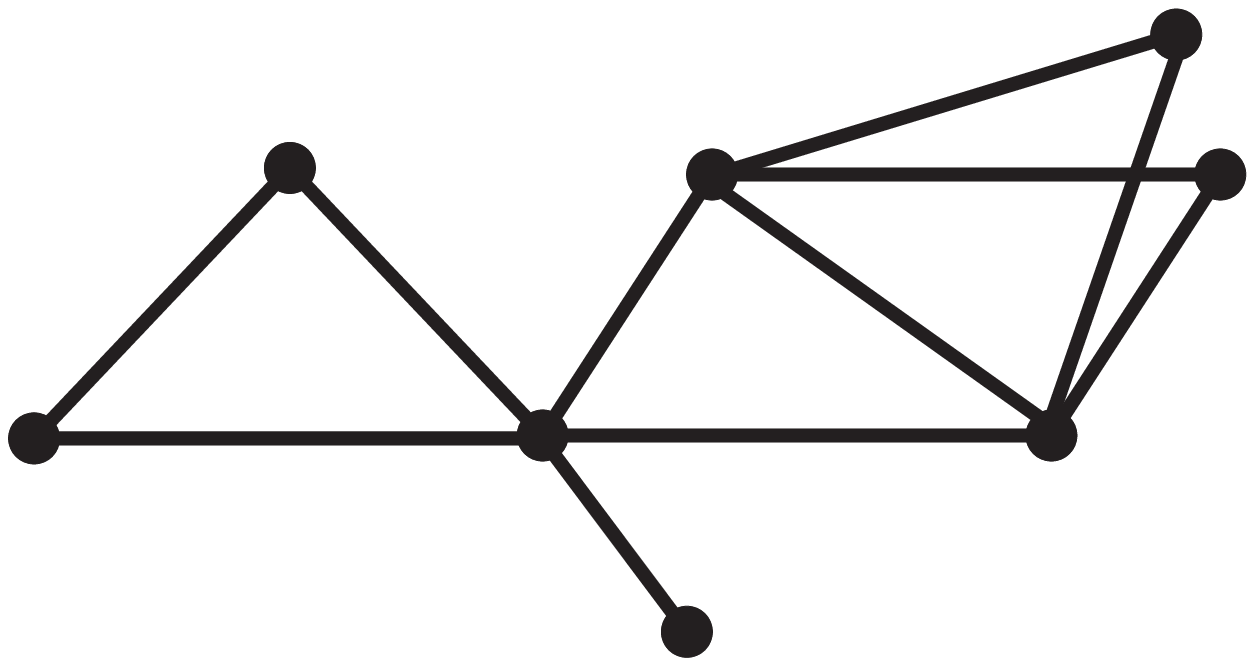}
\caption{\footnotesize{}}\hspace{2 cm}
\label{Generalized}
\end{figure}
\end{center}
Suppose that $G$ is a graph on $[n]$. Let $T$ be a subset of $[n]$, and let $G_1,\ldots,G_{c_G(T)}$ be the connected
components of $G_{[n]\setminus T}$, the induced subgraph of $G$ on $[n]\setminus T$. For each $G_i$, we denote by $\widetilde{G}_i$ the complete graph on the vertex set $V(G_i)$. If
there is no confusion, then we may simply write $c(T)$ instead of $c_G(T)$. Set $$P_T(G)=(\bigcup_{i\in T}\{x_i,y_i\}, J_{\widetilde{G}_1},\ldots,J_{\widetilde{G}_{c(T)}}).$$ Then, $P_T(G)$ is a prime ideal, where $\mathrm{height}\hspace{0.35mm}P_T(G)=n+|T|-c(T)$, by \cite[Lemma~3.1]{HHHKR}. Moreover, $J_G=\bigcap_{T\subset [n]}P_T(G)$, by \cite[Theorem~3.2]{HHHKR}. So that, $\mathrm{dim}\hspace{0.35mm}S/J_G=\mathrm{max}\{n-|T|+c(T):T\subset [n]\}$, by \cite[Cororally~3.3]{HHHKR}. If each $i\in T$ is a cut point of the graph $G_{([n]\setminus T)\cup \{i\}}$, then we say that $T$ has \textbf{cut point
property} for $G$. Let $\mathcal{C}(G)=\{\emptyset\}\cup \{T\subset [n]:T~\mathrm{has~cut~point~property~for}~G\}$. One has $\mathcal{C}(G)=\{\emptyset\}$ if
and only if $G$ is a complete graph. On the other hand, denoted by $\mathcal{M}(G)$, we mean the set of all minimal prime ideals of $J_G$. Then, one has $T\in \mathcal{C}(G)$ if and only if $P_T(G)\in \mathcal{M}(G)$, by \cite[Corollary~3.9]{HHHKR}.

\section{ The Castelnuovo-Mumford regularity of binomial edge ideals }\label{The Castelnuovo-Mumford regularity of binomial edge ideals}

\noindent In this section, we deal with two recent conjectures on the Castelnuovo-Mumford regularity of the binomial edge ideal of a graph, one is due to Matsuda and Murai and the other is due to the authors. Recall that the \textbf{Castelnuovo-Mumford regularity} (or simply, \textbf{regularity}) of a graded $S$-module $M$ is defined as
$$\mathrm{reg}(M)=\mathrm{max}\{j-i~:~\beta_{i,j}(M)\neq 0\}.$$
Denoted by $c(G)$, we mean the number of maximal cliques of the graph $G$. The following are those mentioned conjectures:\\

\noindent \textbf{Conjecture A.} (see \cite{MM}) Let $G$ be a graph on $n$ vertices which is not a path. Then $\mathrm{reg}(J_G)\leq n-1$.\\

\noindent \textbf{Conjecture B.} (see \cite{SK1}) Let $G$ be a graph. Then $\mathrm{reg}(J_G)\leq c(G)+1$.\\

The following bounds for the regularity were given in \cite{MM} and \cite{SK}:

\begin{thm}\label{Kiani-Saeedi}
\cite[Theorem~3.2]{SK} Let $G$ be a closed graph. Then $\mathrm{reg}(J_G)\leq c(G)+1$.
\end{thm}

\begin{thm}\label{Matsuda-Murai}
\cite[Theorem~1.1]{MM} Let $G$ be a graph on $n$ vertices. Then $\mathrm{reg}(J_G)\leq n$.
\end{thm}

Moreover, note that recently, Ene and Zarojanu computed the exact value of the regularity of the binomial edge ideal of a closed graph with respect to the graphical terms in \cite{EZ}, which also yields Theorem~\ref{Kiani-Saeedi}.
Also, by Theorem~\ref{Matsuda-Murai}, it is clear that Conjecture~B is true for trees and unicyclic graphs, whose unique cycles have length greater that three.

Before stating the main theorems, we introduce the notion of the \textit{reduced graph} of a graph.

\begin{defn}
{\em Let $G$ be a graph and $e$ a cut edge of $G$ such that its endpoints are the free vertices of $G\setminus e$. Then, we call $e$, a \textbf{free cut edge} of $G$. Suppose that $\{e_1,\ldots,e_t\}$ is the set of all free cut edges of $G$. Then, we call the graph $G\setminus  \{e_1,\ldots,e_t\}$, the \textbf{reduced graph} of $G$, and denote it by $\mathcal{R}(G)$. If $G$ does not have any free cut edges, then we set $\mathcal{R}(G):=G$. }
\end{defn}

The following are two main theorems of this section:

\begin{thm}\label{main}
Let $G\neq P_n$ be a graph on $n$ vertices which is a disconnected graph, or else \\
{\em{(a)}} has a simplicial vertex, or \\
{\em{(b)}} has a cut edge, or \\
{\em{(c)}} is the join of two graphs, or \\
{\em{(d)}} is an $n$-cycle. \\
Then, we have $\mathrm{reg}(J_G)\leq n-1$.
\end{thm}

\begin{thm}\label{main2}
Suppose that $G$ is a graph such that every connected components of $\mathcal{R}(G)$ is a closed graph, or a generalized block graph. Then, we have $\mathrm{reg}(J_G)\leq c(G)+1$.
\end{thm}

We proceed this section to prove the above theorems. The following theorem might be more general than Theorem~\ref{main2}. 

\begin{thm}\label{construction}
Let $G$ be a connected graph on $n$ vertices and $R_1,\ldots,R_q$ the connected components of $\mathcal{R}(G)$. If Conjecture~B is true for all $R_1,\ldots,R_q$, then it is also true for $G$.
\end{thm}

To prove Theorem~\ref{construction}, we need the following propositions. Here, by $f_e$, we mean the binomial $f_{ij}=x_iy_j-x_jy_i$, where $e=\{i,j\}$ is an edge of $G$.

\begin{prop}\label{colon1}
Let $G$ be a graph and $e$ be a cut edge of $G$. Then we have \\
{\em{(a)}} $\beta_{i,j}(J_G)\leq \beta_{i,j}(J_{G\setminus e})+\beta_{i-1,j-2}(J_{{(G\setminus e)}_e})$, for all $i,j\geq 1$, \\
{\em{(b)}} $\mathrm{pd}(J_G)\leq \mathrm{max}\{\mathrm{pd}(J_{G\setminus e}),\mathrm{pd}(J_{{(G\setminus e)}_e})+1\}$, \\
{\em{(c)}} $\mathrm{reg}(J_G)\leq \mathrm{max}\{\mathrm{reg}(J_{G\setminus e}),\mathrm{reg}(J_{{(G\setminus e)}_e})+1\}$.
\end{prop}

\begin{proof}
It is enough to consider the short exact sequence $$0\longrightarrow S/J_{G\setminus e}:f_e(-2)\stackrel{f_e} \longrightarrow S/J_{G\setminus e}\longrightarrow S/J_G\rightarrow 0.$$ Then, the statements follow by the mapping cone, and the fact $J_{G\setminus e}:f_e=J_{{(G\setminus e)}_e}$, (see \cite[Theorem~3.4]{MSh}).
\end{proof}

Using the above proposition, we partially answer a question that J. Herzog had posed in a discussion with us. He asked if $\beta_{1,j}(J_G)=0$, for all $j>n$, where $G$ is a graph on $n\geq 4$ vertices. Note that in \cite[Theorem~2.2]{SK}, the authors mentioned that $\beta_{1,j}(J_G)=0$, for all $j>2n$.

\begin{cor}\label{Herzog's question}
Let $G$ be a forest on $n\geq 4$ vertices. Then $\beta_{1,j}(J_G)=0$, for all $j>n$.
\end{cor}

\begin{proof}
We use the induction on the number of edges of a forest. If $G$ is a graph with no edges, then $J_G=(0)$, and hence the result is obvious. Now, let $G$ be a forest on $n\geq 4$ vertices and at least one edge. Let $e$ be an arbitrary edge of $G$. So that $e$ is clearly a cut edge of $G$. Thus, by Proposition~\ref{colon1}, for all $j>n$, we have $\beta_{1,j}(J_G)\leq \beta_{1,j}(J_{G\setminus e})+\beta_{0,j-2}(J_{{(G\setminus e)}_e})$. But $\beta_{0,j-2}(J_{{(G\setminus e)}_e})=0$, since $j\geq 5$. Thus, we have $\beta_{1,j}(J_G)\leq \beta_{1,j}(J_{G\setminus e})$, for all $j>n$. By the induction hypothesis, we have $\beta_{1,j}(J_{G\setminus e})=0$, for all $j>n$. Thus, the desired result follows.
\end{proof}

In the statements of the above proposition, the equality does not occur necessarily, because the free resolution which is obtained for $S/J_G$ by mapping cone is not minimal necessarily. But, in a more special case, as in the next proposition, we get the minimal free resolution and hence desired equalities.

\begin{prop}\label{colon2}
Let $G$ be a graph and $e$ be a free cut edge of $G$. Then we have \\
{\em{(a)}} $\beta_{i,j}(J_G)=\beta_{i,j}(J_{G\setminus e})+\beta_{i-1,j-2}(J_{G\setminus e})$, for all $i,j\geq 1$, \\
{\em{(b)}} $\mathrm{pd}(J_G)=\mathrm{pd}(J_{G\setminus e})+1$, \\
{\em{(c)}} $\mathrm{reg}(J_G)=\mathrm{reg}(J_{G\setminus e})+1$.
\end{prop}

\begin{proof}
The proof is as similar as Proposition~\ref{colon1}. Note that $J_{{(G\setminus e)}_e}=J_{G\setminus e}$, since $e$ is a free cut edge of $G$. So, one may consider the short exact sequence $$0\longrightarrow S/J_{G\setminus e}(-2)\stackrel{f_e} \longrightarrow S/J_{G\setminus e}\longrightarrow S/J_G\rightarrow 0$$ instead of that was mentioned in the proof of Proposition~\ref{colon1}. Let $\mathcal{E}$ be the minimal graded free resolution of $S/J_{G\setminus e}$. Now, consider the homomorphism of complexes $\Phi:\mathcal{E}(-2)\longrightarrow \mathcal{E}$, as the multiplication by $f_e$. In fact it is a lift of the map $S/J_{G\setminus e}(-2)\stackrel{f_e} \longrightarrow S/J_{G\setminus e}$. Obviously, the mapping cone over $\Phi$ resolves $S/J_G$. In addition, it is minimal, because $\mathcal{E}$ is minimal and all the maps in the complex homomorphism $\Phi$ are of positive degrees.
\end{proof}

\begin{cor}\label{reg-reduced}
Let $G$ be a connected graph and $R_1,\ldots,R_q$ the connected components of $\mathcal{R}(G)$. Then $\mathrm{reg}(J_G)=\sum_{i=1}^q\mathrm{reg}(J_{R_i})$.
\end{cor}

\begin{proof}
Since $\mathcal{R}(G)$ has $q$ connected components, $G$ has exactly $q-1$ free cut edges. So, by using Proposition~\ref{colon2} repeatedly, we have that $\mathrm{reg}(J_G)=\mathrm{reg}(J_{\mathcal{R}(G)})+q-1$. On the other hand, $\mathrm{reg}(J_{\mathcal{R}(G)})=\sum_{i=1}^{q}\mathrm{reg}(J_{R_i})-q+1$, and hence $\mathrm{reg}(J_G)=\sum_{i=1}^{q}\mathrm{reg}(J_{R_i})$.
\end{proof}

{\em Proof of Theorem~\ref{construction}.}
By Corollary~\ref{reg-reduced}, we have $\mathrm{reg}(J_G)=\sum_{i=1}^q\mathrm{reg}(J_{R_i})$. By the assumption, Conjecture~B is true for $R_i$, for all $i=1,\ldots,q$, so that $\mathrm{reg}(J_{R_i})\leq c(R_i)+1$. Thus, $\mathrm{reg}(J_G)\leq \sum_{i=1}^qc(R_i)+q$. On the other hand, $c(G)=\sum_{i=1}^qc(R_i)+q-1$, because $\mathcal{R}(G)$ has $q$ connected components and hence $G$ has $q-1$ free cut edges. Thus, $\mathrm{reg}(J_G)\leq c(G)+1$, which implies that the conjecture is also true for $G$.  $~~~~~~~~\Box$ \\

Combining Proposition~\ref{colon2} and \cite[Theorem~2.2]{EZ}, we get the following:

\begin{cor}\label{reg-closed-Ene}
Let $G$ be a connected graph and $R_1,\ldots,R_q$ the connected components of $\mathcal{R}(G)$. If $\mathcal{R}(G)$ is closed, then  $$\mathrm{reg}(J_G)=\sum_{i=1}^ql_i+\sharp\{\mathrm{free~cut~edges~of~}G\},$$ where $l_i$ is the length of the longest induced path in $R_i$.
\end{cor}

The following corollary compares the linear strand of $J_G$ and $J_{G\setminus e}$.

\begin{cor}\label{linear strand1}
Let $G$ be a graph and $e$ be a cut edge of $G$. Then $\beta_{i,i+2}(J_G)\leq \beta_{i,i+2}(J_{G\setminus e})$, for all $i\geq 1$.
In particular, if $e$ is a free cut edge of $G$, then $\beta_{i,i+2}(J_G)=\beta_{i,i+2}(J_{G\setminus e})$, for all $i\geq 1$.
\end{cor}

Now, we will get ready to prove the above theorem. We divide the rest of this section into three subsections, each helps us to prove the main theorems of this section.

\subsection{Graphs with a simplicial vertex}\label{Generalized block graphs}

Here, we focus on graphs containing a simplicial vertex. This class of graphs includes a famous class of graphs, i.e. chordal graphs. Indeed, we show that Conjecture~A is true in this case. Moreover, we prove Conjecture~B for a class of chordal graphs containing block graphs. The following is one of the main theorems of this subsection:

\begin{thm}\label{chordal}
Let $G\neq P_n$ be a graph on $[n]$ which contains a simplicial vertex. Then $\mathrm{reg}(J_G)\leq n-1$.
\end{thm}

To prove the above theorem, we need some facts which are mentioned below.

\begin{prop}\label{reg-chordal}
Let $G$ be a graph and $e$ be an edge of $G$. Then we have
$$\mathrm{reg}(J_G)\leq \mathrm{max}\{\mathrm{reg}(J_{G\setminus e}), \reg(J_{G\setminus e}:f_e)+1\}.$$
\end{prop}

\begin{proof}
It suffices to consider the short exact sequence $$0\longrightarrow S/J_{G\setminus e}:f_e(-2)\stackrel{f_e} \longrightarrow S/J_{G\setminus e}\longrightarrow S/J_G\rightarrow 0.$$ Then, the statement follows by applying the mapping cone and \cite[Corollary~18.7]{P}.
\end{proof}

\begin{thm}\label{M-colon1}
\cite[Theorem~3.4]{MSh} Let $G$ be a graph and $e$ be a cut edge of $G$. Then we have $J_{G\setminus e}:f_e=J_{{(G\setminus e)}_e}$.
\end{thm}

\begin{thm}\label{M-colon2}
\cite[Theorem~3.7]{MSh} Let $G$ be a graph and $e=\{i,j\}$ be an edge of $G$. Then we have
$$J_{G\setminus e}:f_e=J_{{(G\setminus e)}_e}+I_G,$$
where $I_G=(g_{P,t}~:~P:i,i_1,\ldots,i_s,j~\mathrm{is~a~path~between~}i,j~\mathrm{in}~G~\mathrm{and}~0\leq t\leq s)$, $g_{P,0}=x_{i_1}\cdots x_{i_s}$ and for every $1\leq t\leq s$, $g_{P,t}=y_{i_1}\cdots y_{i_t}x_{i_{t+1}}\cdots x_{i_s}$.
\end{thm}

\begin{lem}\label{colon3}
Let $G$ be a graph on $[n]$, $v$ a simplicial vertex of $G$ with $\mathrm{deg}_{G}(v)\geq 2$, and $e$ an edge incident with $v$. Then we have $\reg(J_{G\setminus e}:f_e)\leq n-2$.
\end{lem}

\begin{proof}
Let $v_1,\ldots,v_t$ be all the neighbors of the simplicial vertex $v$, and $e_1,\ldots,e_t$ be the edges joining $v$ to $v_1,\ldots,v_t$, respectively, where $t\geq 2$. Without loss of generality, assume that $e:=e_t$. Note that for each $i=1,\ldots,t-1$, $v,v_i,v_t$ is a path between $v$ and $v_t$ in $G\setminus e$, so that for all $i=1,\ldots,t-1$, $x_i$ and $y_i$ are in the minimal monomial set of generators of $I_G$. Also, all other paths between $v$ and $v_t$ in $G\setminus e$ contain $v_i$ for some $i=1,\ldots,t-1$. Thus, all the monomials correspond to these paths, are divisible by either $x_i$ or $y_i$ for some $i=1,\ldots,t-1$. Hence, we have
$I_G=(x_i,y_i:1\leq i\leq t-1)$. So that $J_{G\setminus e}:f_e=J_{{(G\setminus e)}_e}+(x_i,y_i:1\leq i\leq t-1)$. The binomial generators of $J_{{(G\setminus e)}_e}$ correspond to the edges containing vertices $v_1,\ldots,v_{t-1}$, are contained in $I_G$. Let $H:={(G\setminus e)}_e$. Then, we have $J_{G\setminus e}:f_e=J_{H_{[n]\setminus \{v,v_1,\ldots,v_{t-1}\}}}+(x_i,y_i:1\leq i\leq t-1)$, since $v$ is a free vertex of $\Delta(G)$. Thus, $\reg(J_{G\setminus e}:f_e)=\reg(J_{H_{[n]\setminus \{v,v_1,\ldots,v_{t-1}\}}})$. But, $\reg(J_{H_{[n]\setminus \{v,v_1,\ldots,v_{t-1}\}}})\leq n-2$, by Theorem~\ref{M-colon1}, since $t\geq 2$. Therefore, $\reg(J_{G\setminus e}:f_e)\leq n-2$, as desired.
\end{proof}

Now, we go to the proof of Theorem~\ref{chordal}: \\

{\em proof of Theorem~\ref{chordal}.}
We use induction on the number of the vertices. Let $G$ be a graph on $[n]$, with a simplicial vertex, which is not a path. We consider two following cases:

(i) Suppose that $G$ has a simplicial vertex which is a leaf, say $v$. Then, assume that $w$ is the only neighbor of $v$, and $e=\{v,w\}$ is the edge joining $v$ and $w$. We have $\mathrm{reg}(J_{G\setminus e})=\mathrm{reg}(J_{{(G\setminus e)}_{[n]\setminus v}})$, since $v$ is an isolated vertex of $G\setminus e$. Thus, by Theorem~\ref{Matsuda-Murai}, $\mathrm{reg}(J_{G\setminus e})\leq n-1$. On the other hand, we have $\reg(J_{G\setminus e}:f_e)=\reg(J_{{(G\setminus e)}_e})$, by Theorem~\ref{M-colon1}. Note that $v$ is also an isolated vertex of ${(G\setminus e)}_e$, so that we can disregard it in computing the regularity. Thus, $\reg(J_{{(G\setminus e)}_e})\leq n-2$, by the induction hypothesis, since ${(G\setminus e)}_e$ has $w$ as a simplicial vertex. Hence, $\reg(J_{G\setminus e}:f_e)+1\leq n-1$. Thus, by Proposition~\ref{reg-chordal}, we get $\reg(J_G)\leq n-1$.

(ii) Suppose that all the simplicial vertices of $G$ have degree greater than one. Let $v$ be a simplicial vertex of $G$ and $v_1,\ldots,v_t$ be all the neighbors of $v$, and $e_1,\ldots,e_t$ be the edges joining $v$ to $v_1,\ldots,v_t$, respectively, where $t\geq 2$. Using repeatedly Proposition~\ref{reg} and Lemma~\ref{colon3}, we get $\reg(J_G)\leq \mathrm{max}\{\reg(J_{G\setminus \{e_1,\ldots,e_{t-1}\}}),n-1\}$. Note that $G\setminus \{e_1,\ldots,e_{t-1}\}$ is a graph on $n$ vertices in which $v$ is a leaf. Thus, by case~(i), we have $\reg(J_{G\setminus \{e_1,\ldots,e_{t-1}\}}))\leq n-1$, by case~(i). Thus, $\reg(J_G)\leq n-1$.

Therefore, by the above cases, we get the desired result.  ~~~$~~~~~~~~\Box$ \\

Now, recall that a facet $F$ of a simplicial complex $\Delta$ is
called a \textbf{leaf}, if either $F$ is the only facet, or there exists a facet $G$, called a \textbf{branch} of $F$, such that for each facet $H$ of $\Delta$, with $H\neq F$, one has $H\cap F \subseteq G\cap F$. One can see that each leaf $F$ has at least a free vertex. A simplicial complex $\Delta$ is called a \textbf{quasi-forest} if its facets can be ordered as $F_1,\ldots,F_r$ such that for all $i>1$, $F_i$ is a leaf of $\Delta$ with facets $F_1,\ldots,F_{i-1}$. Such an order of the facets is called a \textbf{leaf order}. A connected quasi-forest is called a \textbf{quasi-tree}.

The following theorem extends \cite[Theorem~2.9~and~Corollary~2.10]{EZ} to a wider class of chordal graphs.

\begin{thm}\label{reg-generalized}
Let $G$ be a generalized block graph on $[n]$. Then $$\mathrm{reg}(J_G)\leq c(G)+1.$$
\end{thm}

\begin{proof}
Our proof is similar to the proof of \cite[Theorem~2.9]{EZ}, which is based on the technique applied in the proof of \cite[Theorem~1.1]{EHH}. So, we omit some details. By Dirac's theorem (see \cite{D}), $\Delta(G)$ is a quasi tree, since $G$ is connected and chordal. Let $c:=c(G)$, and $F_1,\ldots,F_c$ be a leaf order of the facets of $\Delta(G)$. We use induction on $c$, the number of maximal cliques of $G$. If $c=1$, then the result is obvious. Let $c>1$ and $F_{t_1},\ldots,F_{t_q}$ be all the branches of the leaf $F_c$. Since $G$ is a generalized block graph, each pair of the facets $F_c,F_{t_1},\ldots,F_{t_q}$ intersect in exactly the same set of vertices, say $A$, and also, $F_c\cap F_l=\emptyset$, for all $l\neq t_1,\ldots,t_q$, as $F_c$ is a leaf. Hence, we have that $A\cap F_l=\emptyset$, for all $l\neq t_1,\ldots,t_q$. On the other hand, for $T\subset [n]$ with $T\in \mathcal{C}(G)$, we have that $A\nsubseteq T$ if and only if $A\cap T=\emptyset$; because if $|A|>1$ and $v\in A\cap T$, then $v$ is not a cut point of the graph $G_{([n]\setminus T)\cup \{v\}}$, as $A\setminus T\neq \emptyset$, so it is a contradiction. Thus, let $J_G=Q\cap Q'$, where
 \begin{equation}
Q=\bigcap_{\substack{
T\in \mathcal{C}(G) \\
A\cap T=\emptyset
}}
P_T(G)~~,~~Q'=\bigcap_{\substack{
T\in \mathcal{C}(G) \\
A\subseteq T
}}
P_T(G).
\nonumber
\end{equation}
Similar to the proof of \cite[Theorem~2.9]{EZ} and \cite[Theorem~1.1]{EHH}, let $G'$ be the graph obtained from $G$, by replacing the cliques $F_c,F_{t_1},\ldots,F_{t_q}$, by a clique on the vertex set $F_c\cup (\bigcup_{j=1}^qF_{t_j})$. One can see that $Q=J_{G'}$ and $Q'=(x_i,y_i:i\in A)+J_{G_{[n]\setminus A}}$. Thus, we have $Q+Q'=(x_i,y_i:i\in A)+J_{{G'}_{[n]\setminus A}}$. By the definition of generalized block graphs, it is not difficult to see that $G'$, $G_{[n]\setminus A}$ and ${G'}_{[n]\setminus A}$ are also generalized block graphs. Then, considering the short exact sequence $$0\rightarrow J_G\rightarrow Q\oplus Q'\rightarrow Q+Q'\rightarrow 0,$$ the result follows, as in \cite[Theorem~2.9]{EZ}.
\end{proof}

\subsection{Join of Graphs}\label{Join of Graphs}

Here, we focus on the join of two graphs and deal with the conjectures for this class of graphs. Consequently, we gain some results on complete $t$-partite graphs, which generalize some previous results.

Note that the join of two complete graphs is also obviously complete, so that its binomial edge ideal has a linear resolution, by \cite[Theorem~2.1]{SK}, and hence its regularity is equal to $2$. The following theorem determines the regularity of the binomial edge ideal of the join of two graphs with respect to the original graphs', when they are not both complete graphs.

\begin{thm}\label{reg-join}
Let $G_1$ and $G_2$ be graphs on $[n_1]$ and $[n_2]$, respectively, not both complete. Then
$$\mathrm{reg}(J_{G_1*G_2})=\mathrm{max}\{\mathrm{reg}(J_{G_1}),\mathrm{reg}(J_{G_2}),3\}.$$
\end{thm}

To prove the above theorem, we need some facts which are mentioned in the sequel. If $H$ is a graph with connected components $H_1,\ldots,H_r$, then we denote it by $\bigsqcup_{i=1}^r H_i$.

\begin{prop}\label{both disconnected 1}
Suppose that $G_1=\bigsqcup_{i=1}^r G_{1i}$ and $G_2=\bigsqcup_{i=1}^s G_{2i}$ are two graphs on disjoint sets of vertices
$[n_1]=\bigcup_{i=1}^r [n_{1i}]$ and $[n_2]=\bigcup_{i=1}^s [n_{2i}]$,
respectively, where $r,s\geq 2$. Then we have
$$\mathcal{C}(G_1*G_2)=\{\emptyset\}\cup \big{(}(\bigcirc_{i=1}^{r}\mathcal{C}(G_{1i}))\circ \{[n_2]\} \big{)}\cup \big{(}(\bigcirc_{i=1}^{s}\mathcal{C}(G_{2i}))\circ \{[n_1]\} \big{)}.$$
\end{prop}

\begin{proof}
Let $G:=G_1*G_2$ and $T\in (\bigcirc_{i=1}^{r}\mathcal{C}(G_{1i}))\circ \{[n_2]\}$. So, $T=[n_2]\cup (\bigcup_{i=1}^r T_{1i})$, where $T_{1i}\in \mathcal{C}(G_{1i})$, for $i=1,\ldots,r$. We show that $T$ has cut point property. Let $j\in T$. If $j\in T_{1i}$, for some $i=1,\ldots,r$, then $G_{([n]\setminus T)\cup \{j\}}={G_{1i}}_{([n_{1i}]\setminus T_{1i})\cup \{j\}}\sqcup (\bigsqcup_{l=1,l\neq i}^r {G_{1l}}_{([n_{1l}]\setminus T_{1l})})$. In this case, $j$ is a cut point of ${G_{1i}}_{([n_{1i}]\setminus T_{1i})\cup \{j\}}$, since $T_{1i}\in \mathcal{C}(G_{1i})$. So that $j$ is also a cut point of $G_{([n]\setminus T)\cup \{j\}}$. If $j\in [n_2]$, then $G_{([n]\setminus T)\cup \{j\}}=j*\bigsqcup_{i=1}^r {G_{1i}}_{([n_{1i}]\setminus T_{1i})}$. So, $j$ is a cut point of $G_{([n]\setminus T)\cup \{j\}}$, since $G_{([n]\setminus T)}$ is disconnected. Thus, in both cases, $T$ has cut point property. If $T\in (\bigcirc_{i=1}^{s}\mathcal{C}(G_{2i}))\circ \{[n_1]\}$, then similarly, we have $T\in \mathcal{C}(G)$.
For the other inclusion, let $\emptyset \neq T\in \mathcal{C}(G)$. If $T$ does not contain $[n_1]$ and $[n_2]$, then $G_{[n]\setminus T}$ is connected, and hence no element $i$ of $T$ is a cut point of $G_{([n]\setminus T)\cup \{i\}}$. So, we have $[n_1]\subseteq T$ or $[n_2]\subseteq T$. Suppose that $[n_1]\subseteq T$. Then, $T=[n_1]\cup (\bigcup_{i=1}^s T_{2i})$, where $T_{2i}\subseteq [n_2]$, for $i=1,\ldots,s$. Let $1\leq i\leq s$. If $T_{2i}=\emptyset$, then, clearly, $T_{2i}\in \mathcal{C}(G_{2i})$. If $T_{2i}\neq \emptyset$, then
each $j\in T_{2i}$, is a cut point of $G_{([n]\setminus T)\cup \{j\}}$, since $T\in \mathcal{C}(G)$. So that $j$ is a cut point of ${G_{2i}}_{([n_{2i}]\setminus T_{2i})\cup \{j\}}$, because $j\in T_{2i}$ and $G_{2i}$'s are on disjoint sets of vertices. Thus, $T_{2i}\in \mathcal{C}(G_{2i})$. Therefore, $T\in (\bigcirc_{i=1}^{s}\mathcal{C}(G_{2i}))\circ \{[n_1]\}$. If $[n_2]\subseteq T$, then similarly we get $T\in (\bigcirc_{i=1}^{r}\mathcal{C}(G_{1i}))\circ \{[n_2]\}$.
\end{proof}

The following is a corollary of a result in \cite{SK1} about induced subgraphs of a graph:

\begin{prop}\label{induced}
\cite[Proposition~8]{SK1} Let $G$ be a graph and $H$ an induced subgraph of $G$. Then we have \\
{\em{(a)}} $\beta_{i,j}(J_{H})\leq \beta_{i,j}(J_{G})$, for all $i,j$.\\
{\em{(b)}} $\mathrm{reg}(J_{H})\leq \mathrm{reg}(J_{G})$. \\
{\em{(c)}} $\mathrm{pd}(J_{H})\leq \mathrm{pd}(J_{G})$.
\end{prop}

{\em Proof of Theorem~\ref{reg-join}.} Let $G:=G_1*G_2$. Note that since $G$ is not a complete graph, $J_G$ does not have a linear resolution, by \cite[Theorem~2.1]{SK}. So that $\mathrm{reg}(J_G)\geq 3$. On the other hand, by Proposition~\ref{induced}, $\mathrm{reg}(J_G)\geq\mathrm{reg}(J_{G_1})$ and $\mathrm{reg}(J_G)\geq\mathrm{reg}(J_{G_2})$, because $G_1$ and $G_2$ are induced subgraphs of $G$. So, $\mathrm{reg}(J_G)\geq \mathrm{max}\{\mathrm{reg}(J_{G_1}),\mathrm{reg}(J_{G_2}),3\}$. For the other inequality, first, suppose that $G_1$ and $G_2$ are both disconnected graphs. Let $G_1=\bigsqcup_{i=1}^r G_{1i}$ and $G_2=\bigsqcup_{i=1}^s G_{2i}$ be two graphs on disjoint sets of vertices
$[n_1]=\bigcup_{i=1}^r [n_{1i}]$ and $[n_2]=\bigcup_{i=1}^s [n_{2i}]$,
respectively, where $r,s\geq 2$. By Proposition~\ref{both disconnected 1}, $\mathcal{C}(G)=\{\emptyset\}\cup \big{(}(\bigcirc_{i=1}^{r}\mathcal{C}(G_{1i}))\circ \{[n_2]\} \big{)}\cup \big{(}(\bigcirc_{i=1}^{s}\mathcal{C}(G_{2i}))\circ \{[n_1]\} \big{)}$. So, $J_{G}=Q\cap Q'$, where
\begin{equation}
Q=\bigcap_{\substack{
T\in \mathcal{C}(G) \\
[n_1]\subseteq T
}}
P_T(G)~~,~~Q'=\bigcap_{\substack{
T\in \mathcal{C}(G) \\
[n_1]\nsubseteq T
}}
P_T(G).
\nonumber
\end{equation}
Thus, we have
\begin{equation}
Q=(x_i,y_i:i\in [n_1])+\bigcap_{\substack{
T\in \mathcal{C}(G) \\
[n_1]\subseteq T
}}
P_{T\setminus [n_1]}(G_2)
\nonumber
\end{equation}
and
\begin{equation}
Q'=P_{\emptyset}(G)\cap\big{(}\bigcap_{\substack{
\emptyset\neq T\in \mathcal{C}(G) \\
[n_1]\nsubseteq T
}}
P_T(G)\big{)}
=P_{\emptyset}(G)\cap \big{(}(x_i,y_i:i\in [n_2])+\bigcap_{\substack{
T\in \mathcal{C}(G) \\
[n_2]\subseteq T
}}
P_{T\setminus [n_2]}(G_1)\big{)}.
\nonumber
\end{equation}
So, one can see that $Q=(x_i,y_i:i\in [n_1])+J_{G_2}$, $Q'=J_{K_n}\cap \big{(}(x_i,y_i:i\in [n_2])+J_{G_1}\big{)}$ and $Q+Q'=(x_i,y_i:i\in [n_1])+J_{K_{n_2}}$.
Now, consider the short exact sequence $$0\rightarrow J_G\rightarrow Q\oplus Q'\rightarrow Q+Q'\rightarrow 0.$$
By \cite[Corollary~18.7]{P}, we have $\mathrm{reg}(J_G)\leq \mathrm{max}\{\mathrm{reg}(Q),\mathrm{reg}(Q'),\mathrm{reg}(Q+Q')+1\}$. On the other hand, we have $\mathrm{reg}(Q)=\mathrm{reg}(J_{G_2})$, $\mathrm{reg}(Q')\leq \mathrm{max}\{\mathrm{reg}(J_{G_1}),\mathrm{reg}(K_{n_1})+1=3\}$ (by using a suitable short exact sequence as above), and $\mathrm{reg}(Q+Q')=\mathrm{reg}(K_{n_2})+1=3$. Hence, $\mathrm{reg}(J_G)\leq \mathrm{max}\{\mathrm{reg}(J_{G_2}),\mathrm{reg}(J_{G_1}),3\}$. Now, suppose that $G_1$ or $G_2$ is connected. We add an isolated vertex $v$ to $G_1$ and an isolated vertex $w$ to $G_2$. Thus, we obtain two disconnected graphs $G_1'$ and $G_2'$. So, by the above discussion, we have $\mathrm{reg}(J_{G_1'*G_2'})\leq \mathrm{max}\{\mathrm{reg}(J_{G_1'}),\mathrm{reg}(J_{G_2'}),3\}$. But, clearly, we have $\mathrm{reg}(J_{G_1'})=\mathrm{reg}(J_{G_1})$ and $\mathrm{reg}(J_{G_2'})=\mathrm{reg}(J_{G_2})$, so that $\mathrm{reg}(J_{G_1'*G_2'})\leq \mathrm{max}\{\mathrm{reg}(J_{G_1}),\mathrm{reg}(J_{G_2}),3\}$. Thus, the result follows by Proposition~\ref{induced}, since $G_1*G_2$ is an induced subgraph of $G_1'*G_2'$. ~~~$~~~~~~~~\Box$

\begin{rem}\label{Sharpness}
{\em By Theorem~\ref{reg-join}, we have if $G$ is a (multi)-fan graph (i.e. $K_1*\bigsqcup_{i=1}^t P_{n_i}$, for some $t\geq 1$, which might be a non-closed graph), then $\mathrm{reg}(J_G)=c(G)+1$. This implies that if Conjecture~B is true, then the given bound is sharp. }
\end{rem}

\begin{cor}\label{ConjB-join}
Let $G_1$ and $G_2$ be two graphs on $[n_1]$ and $[n_2]$, respectively. If Conjecture~B is true for $G_1$ and $G_2$, then it is also true for $G_1*G_2$.
\end{cor}

\begin{proof}
It is enough to note that $c(G_1*G_2)=c(G_1)c(G_2)$, and if $G_1$ and $G_2$ are complete graphs, then $G_1*G_2$ is also complete and Conjecture~B is true for it.
\end{proof}

The following corollary proves the conjecture of Matsuda and Murai in the case of join of graphs:

\begin{cor}\label{ConjA-join}
Let $G$ be a graph on $n$ vertices which is the join of two graphs. If $G$ is not $P_2$ nor $P_3$, then $\mathrm{reg}(J_{G})\leq n-1$.
\end{cor}

\begin{proof}
It is enough to apply Theorem~\ref{reg-join} and Theorem~\ref{Matsuda-Murai}.
\end{proof}

The following corollary generalizes the result of \cite{SZ} on the regularity of complete bipartite graphs:

\begin{cor}\label{reg-t-partite}
Let $G$ be a complete $t$-partite graph, where $t\geq 2$. If $G$ is not complete, then $\mathrm{reg}(J_{G})=3$. In particular, Conjecture~A is true for complete $t$-partite graphs.
\end{cor}

\begin{proof}
We use induction on $t\geq 2$, the number of parts. If $t=2$, then $G$ is the join of two graphs each consisting of some isolated vertices. So, the regularity of the binomial edge ideal of each of them is $0$. Thus, by Theorem~\ref{reg-join}, we have $\mathrm{reg}(J_{G})=3$. Now, suppose that $t>2$ and the result is true for every complete $(t-1)$-partite graph which is not complete. Let $V_1,\ldots,V_t$ be the partition of the vertices of $G$ to $t$ parts. Hence, we have
$G=G_{V_t}*G_{V\setminus V_t}$, and $G_{V\setminus V_t}$ is a complete $(t-1)$-partite graph. If $G_{V\setminus V_t}$ is a complete graph, then $|V_t|>1$, since, otherwise, $G$ is a complete graph, a contradiction. So, by Theorem~\ref{reg-join}, $\mathrm{reg}(J_{G})=3$. If $G_{V\setminus V_t}$ is not complete, then by the induction hypothesis, we have $\mathrm{reg}(J_{G_{V\setminus V_t}})=3$. Thus, again by Theorem~\ref{reg-join}, the result follows.
\end{proof}

\subsection{Proof of the main theorems}\label{Proof}

Now, we go to the proof of our main theorems. \\

{\em Proof of Theorem~\ref{main}.} If $G$ is a disconnected graph with $r\geq 2$ connected components $H_1,\ldots,H_r$, over $n_1,\ldots,n_r$ vertices, respectively, then we have $\mathrm{reg}(J_G)=\sum_{i=1}^{r}\mathrm{reg}(J_{H_i})-r+1$. By Theorem~\ref{Matsuda-Murai}, we have $\mathrm{reg}(J_{H_i})\leq n_i$. So that $\mathrm{reg}(J_G)\leq \sum_{i=1}^rn_i-r+1=n-r+1\leq n-1$, since $r\geq 2$. So, the result follows in this case. Now, suppose that $G$ has a cut edge $e$. Then, $G\setminus e$ is disconnected and hence $\mathrm{reg}(J_{G\setminus e})\leq n-1$. On the other hand, ${(G\setminus e)}_e$ has two connected components, say $G_1$ and $G_2$, which both have a simplicial vertex. So that Conjecture~A is true for both of them, by Theorem~\ref{chordal}. Note that $G_1$ and $G_2$ are not both paths, as $G$ is not. Thus, as mentioned above, we have $\mathrm{reg}(J_{{(G\setminus e)}_e})=\mathrm{reg}(J_{G_1})+\mathrm{reg}(J_{G_2})-1\leq n-2$. So, by Proposition~\ref{colon2}, we get the result in this case too. Now, combining Theorem~\ref{reg-chordal}, Theorem~\ref{reg-join} and \cite[Corollary~3.8]{ZZ} we get the result. $~~~~~~~~\Box$
\\

{\em Proof of Theorem~\ref{main2}.} By Theorem~\ref{construction}, Theorem~\ref{Kiani-Saeedi} and Theorem~\ref{reg-generalized}, the result follows.
$~~~~~~~~\Box$ \\

\textbf{Acknowledgments:} The authors would like to thank Professor J. Herzog for some useful comments. Also, the authors would like to thank to the Institute for Research in Fundamental Sciences (IPM) for financial support. The research of the first author was in part supported by a grant from IPM (No. 92050220).

\providecommand{\byame}{\leavevmode\hbox
to3em{\hrulefill}\thinspace}

\end{document}